\documentclass[a4paper,11pt,reqno]{amsart}
\usepackage[utf8]{inputenc}
\usepackage[T1]{fontenc}
\usepackage{lmodern}
\usepackage[english]{babel}
\usepackage{amsmath,a4wide}
\usepackage{xfrac}
\usepackage{esint}
\usepackage{stmaryrd,mathrsfs,bm,amsthm,mathtools,yfonts,amssymb,color,braket,booktabs,graphicx,graphics,amsfonts}
\usepackage{latexsym,microtype,indentfirst,hyperref}
\usepackage{xcolor}
\usepackage{courier}
\usepackage[colorinlistoftodos]{todonotes}

\begingroup
\newtheorem{theorem}{Theorem}[section]
\newtheorem{lemma}[theorem]{Lemma}
\newtheorem{proposition}[theorem]{Proposition}
\newtheorem{corollary}[theorem]{Corollary}

\endgroup

\theoremstyle{definition}
\newtheorem{definition}[theorem]{Definition}
\newtheorem{remark}[theorem]{Remark}

\newcommand\res{\mathop{\hbox{\vrule height 7pt width .3pt depth 0pt\vrule height .3pt width 5pt depth 0pt}}\nolimits}

\newcommand{\HH}{\mathcal{H}}
\newcommand{\diam}{\text{diam}}
\newcommand{\dist}{\text{dist}}

\title[An elementary rectifiability lemma]{An elementary rectifiability lemma and some applications}

\author[C. De Lellis]{Camillo De Lellis}
\address{School of Mathematics, Institute for Advanced Study, 1 Einstein Dr., Princeton NJ 08540, USA}
\email{camillo.delellis@ias.edu}

\author[I. Fleschler]{Ian Fleschler}
\address{Department of Mathematics, Fine Hall, Princeton University, Washington Road, Princeton, NJ 08540, USA}
\email{imf@princeton.edu}

\begin{document}

\maketitle

\begin{abstract}
We generalize a classical theorem of Besicovitch, showing that, for any positive integers $k<n$, if $E\subset \mathbb R^n$ is a Souslin set which is not $\mathcal{H}^k$-$\sigma$-finite, then $E$ contains a purely unrectifiable closed set $F$ with $0< \mathcal{H}^k (F) < \infty$. Therefore, if $E\subset \mathbb R^n$ is a Souslin set with the property that every closed subset with finite $\mathcal{H}^k$ measure is $k$-rectifiable, then $E$ is $k$-rectifiable. We also point out that this theorem holds in a suitable class of metric spaces. Our interest is motivated by recent studies of the structure of the singular sets of several objects in geometric analysis and we explain the usefulness of our lemma with some examples. 
\end{abstract}

\section{Introduction}

The purpose of this note is to show the following lemma in geometric measure theory and to show some interesting applications of it. As usual $\mathcal{H}^s$ denotes the Hausdorff $s$-dimensional measure, while for the definition of Souslin sets (also called analytic sets) we refer to \cite[Section 2.2.10]{Federer}.

\begin{theorem}\label{t:main} Let $1\leq k < n$ be two given integers.
Assume $E\subset \mathbb R^n$ is a Souslin set which is not $\mathcal{H}^k$ $\sigma$-finite. Then $E$ contains a closed subset $F$ such that $0 < \mathcal{H}^k (F)< \infty$ and which is purely $k$-unrectifiable, namely $\mathcal{H}^k (F\cap \Gamma)=0$ for every Lipschitz $k$-dimensional graph $\Gamma$. 
\end{theorem}

In our proof the set $F$ is the support of a nontrivial Frostman measure $\mu$, namely a nonnegative Radon measure $\mu$ with the property that 
\begin{equation}\label{e:Frostman}
\mu (B_r (x)) \leq r^k \qquad \mbox{for every $x$ and every $r$}.
\end{equation}
As a simple corollary we then get the following

\begin{corollary}\label{c:main} Let $1\leq k < n$ be two given integers.
Assume $E\subset \mathbb R^n$ is a Souslin subset and assume that any nonnegative Radon measure satisfying \eqref{e:Frostman}, supported in $E$, and with $\liminf_{r\downarrow 0} r^{-k} \mu (B_r (x)) > 0$ for $\mu$-a.e. $x$ is $k$-rectifiable (namely there is a $k$-rectifiable set $R$ and a Borel function $f$ such that $\mu = f \mathcal{H}^k \res R$). Then $E$ is $k$-rectifiable.
\end{corollary} 

\begin{remark}\label{r:doubling}
Since any Borel set is a Souslin set, the theorem and the corollary apply to Borel sets $E$.
\end{remark}

\begin{remark}\label{r:metric}
The theorem also holds when $E \subset (X,d)$ is a Souslin subset of a $\sigma$-compact doubling metric space, i.e. a metric space for which there exists $M$ with the property that every ball of radius $r$ can be covered by at most $M$ balls of radius $r/2$. Since our main interest is in the Euclidean space we will only briefly sketch how to handle this more general case.
\end{remark}

The case $n=2$ and $k=1$ of Theorem \ref{t:main} is contained in a classical work of Besicovitch, cf. \cite[Theorem 6]{Besicovitch}. However, in spite of its rather elementary nature, we have not been able to find Theorem \ref{t:main} in the literature, nor we have found any work asking the fairly obvious question of whether Besicovitch's theorem can be generalized to arbitrary dimension and codimension. 

\subsection{Rectifiability questions in geometric analysis} Our interest has been partially spurred by the groundbreaking work of Naber and Valtorta \cite{NV}.
In several problems of geometric nature, for instance in the theory of minimal surfaces, that of harmonic maps, or that of Ricci limits, one is naturally lead to study solutions of PDEs, or minimal/critical points of variational problems, or generalized spaces, which are not everywhere smooth, but for which a suitable regularity theory gives upper bounds on the dimension of the singular set (defined as the complement of all points at which the object of interest is smooth). When the upper bound is optimal, namely it matches the dimension of the singular set for some known examples, the next natural question is whether the singular set has some more structure. In some of the examples mentioned above, i.e. in the case of singularities of harmonic maps and of minimal surfaces, showing even just the rectifiability of the singular set has been surprisingly hard. In some notable cases (minimizing harmonic maps, area-minimizing hypersurfaces, and mod-$2$ minimizing surfaces in any dimension and codimension) this was achieved in pioneering works of Leon Simon in the nineties. We refer to chapter 15 of \cite{MattilaSurvey} for a more detailed overview of questions on the rectifiability of singularities of solutions to geometric variational problems.

A few years ago Naber and Valtorta introduced in \cite{NV} a rather powerful and flexible technique to recover and improve Simon's rectifiability results. The approach of \cite{NV} has in fact been extended to several other situations, especially due to its flexibility (cf. for instance \cite{O1,O2,CJN,DMSV,DS,ENV,FS1,FS2,FS3,NV2,NV3}): while Simon's proof uses quite hard PDE techniques, the approach of Naber and Valtorta uses very little of the problem at hand, the real key point being the availability of a monotonicity formula with some suitable structure, a property which is common to a variety of situations in geometric analysis. 

Roughly speaking the approach of \cite{NV} can be subdivided into three steps. To fix ideas we assume that the singular set we are interested in has Hausdorff dimension $k$. 
\begin{itemize}
\item[(i)] A first general theorem, which has been independently discovered by Azzam and Tolsa in \cite{AT}, ensures the $k$-rectifiability of a $k$-dimensional measure $\mu$ under a suitable sharp control of what in the literature is called $L^2$ $\beta$-number. For a particular case of the statement see Theorem \ref{t:ANTV} below (the version of \cite{ENV} is slightly stronger, the version we refer to was already proved combining \cite{AT} and \cite{T}). 
\item[(ii)] The sharp control needed in (i) is then achieved for every Frostman $k$-dimensional measure $\mu$ which is supported in the singular set using in a careful and clever way the monotonicity formula available for the problem.
\item[(iii)] If one knew the $\sigma$-finiteness of the singular set with respect to the $\mathcal{H}^k$ measure, (i) and (ii) would then immediately imply its rectifiability. However we only know a priori that its Hausdorff dimension is $k$. To overcome this difficulty, Naber and Valtorta use quite subtle covering arguments and an approximate version of (i)-(ii) for suitably discretized measures. An alternative approach is given in \cite{CJN} through what the authors call ``neck regions''
\end{itemize}
The reason of our interest in Theorem \ref{t:main} should now be obvious: with the latter at hand, we can just completely bypass point (iii) in the Naber-Valtorta strategy, and conclude immediately the $\sigma$-finiteness and rectifiability from (i) and (ii) at an ``abstract level''. It must be however noted that, while Theorem \ref{t:main} is thus a useful general tool, which allows to bypass a quite sizable and difficult argument in (iii), it does not have its exact same power. In fact, under certain additional assumptions, Naber and Valtorta are able to use their arguments in step (iii) to conclude also the local $\mathcal{H}^k$-{\em finiteness} of the singular set. We do not exclude the possibility that, once we know the rectifiability through Theorem \ref{t:main}, the argument of Naber-Valtorta for proving $\mathcal{H}^k$ local finiteness could be simplified and streamlined: we have simply not tried to do it.

For completeness we end this paragraph by stating the Azzam-Naber-Tolsa-Valtorta Theorem referenced in point (i) above.

\begin{theorem}\label{t:ANTV}
Let $S\subset \mathbb R^n$ be an $\mathcal{H}^k$-measurable set with $0 <\mathcal{H}^k (S) < \infty$ and consider $\mu:= \mathcal{H}^k \res S$. Then $S$ is $k$-rectifiable if and only if
\[
\int_0^1 \beta_{2,\mu}^k (x,s)^2\, \frac{ds}{s} < \infty \qquad \mbox{for $\mu$-a.e. $x$,}
\]
where 
\[
\beta_{2,\mu}^k (x,s) := \inf \left(\left\{r^{-k-2} \int_{B_r (x)} \dist^2 (y, L)\, d\mu (y)\right)^{1/2}: \mbox{ $L$ is a $k$-dim. affine space}\right\}\, .
\]
\end{theorem}

\subsection{Quantitative rectifiability} We are also interested in Theorem \ref{t:main} in the broader context of quantitative rectifiability. In quantitative rectifiability many questions (which a posteriori can have many fruitful applications) can be phrased in the following form:
\begin{itemize}
\item[(Q)] Consider a quantity $\alpha_E(x,r)$ which measures ``how close to $k$-rectifiable a set $E$ is'' (such as some form of flatness or symmetry) at the scale $r$ around the point $x$. Having fixed $E$ consider
\begin{equation*}
A_{\alpha_{E}}:=\left \{x \in E: \int_0^1 \alpha_E(x,r)^2 \, \frac{dr}{r} <\infty\right\}.
\end{equation*}
Is $A$ rectifiable? More specifically, if we knew that $A_{\alpha_{E}}=E$, is $E$ rectifiable? 
\end{itemize}
The idea is then to use Theorem \ref{t:main}, in appropriate situations, to reduce the question above to the much more favorable case where $E$ has finite $\mathcal{H}^k$ measure and $A_{\alpha_{E}}=E$. 

For concreteness we will give two examples.
We define the $\beta$ coefficient for the Hausdorff content as:
\begin{equation*}
\widetilde{\beta}^k_{2,E}(x,r):= \inf \left\{\left(r^{-k-2} \int_{B_r (x)\cap E} \dist^2 (y, L)\, d\mathcal{H}_{\infty}^{k}(y)\right)^{1/2}: \mbox{ $L$ is a $k$-dim. affine space}\right\}\ .
\end{equation*}
The integral above is with respect to the Hausdorff content $\mathcal{H}^k_\infty$, which is {\em not a measure}. In particular we define such integral as
\begin{equation*}
    \int_A f(x)^p d\mathcal{H}_{\infty}^k(x):=\int_{0}^\infty \mathcal{H}_{\infty}^k\left(\{x \in A | f(x)>t\} \right)t^{p-1}dt.
\end{equation*}
Let $E$ be a compact set. We first remark that $A_{\widetilde{\beta}_{2,E}^k}$ is Borel. The coefficient $\widetilde{\beta}_{2,E}$ is upper semicontinous on $(x,r)$ and the integral can be equivalently taken as a discrete sum over dyadic scales. This information is enough to verify that the set is Borel since the coefficient is a Borel measurable function of $x$.

Let $\mu$ be any Frostman (or $k$ upper regular) measure supported on $E$. It is rather easy to see then that:
\begin{equation*}
    \beta_{2,\mu}^k(x,r) \lesssim \widetilde{\beta}_{2,E}^k(x,r).
\end{equation*}
This means that, if $F \subseteq A_{\widetilde{\beta}_{2,E}^k}$ is a Borel set such that $\mu=c\mathcal{H}^k\res F$ (for some constant $c>0$) is Frostman, then $\mu$ has finite $\beta_{2,\mu}^k$ square function. In particular we can use Theorem \ref{t:ANTV} to say that every such subset is rectifiable. Fix now any $F\subset A_{\widetilde{\beta}_{2,E}^k}$ with finite Hausdorff measure. A simple exercise shows that $F$ can be written as a countable union $Z\cup \bigcup_i F_i$, where the Borel sets $F_i$ enjoy the above property and $Z$ is $\mathcal{H}^k$-null. We can apply Theorem \ref{t:main} and conclude that the answer to (Q) is positive. 
The coefficients $\widetilde{\beta}_{2,E}^k$ were introduced by Azzam and Schul in their work \cite{AS} on the analyst's traveling salesman theorem in the context of lower content regular sets.
The study of the set $A_{\widetilde{\beta}_{2,E}^k}$ has been done in the context of tangent points to lower content regular sets by Villa in \cite{V}. The argument above gives that the latter set is rectifiable even when the lower content regularity is dropped.

The second author was also interested in Theorem \ref{t:main} due to his work on the higher dimensional Carleson $\varepsilon^2$ conjecture, joint with Tolsa and Villa.  Let  $\Omega^+,\Omega^-\subset \mathbb{R}^{n+1}$ be open and disjoint and let 
$B(x,r)\subset \mathbb{R}^{n+1}$ be a ball. Let $H\subset\mathbb{R}^{n+1}$ be a halfspace such that $x\in \partial H$. 
Denote $S_{H}^+ = \partial B(x,r) \cap H$ and $S_{H}^- = \partial B(x,r) \setminus \overline{H}$.
Define 
\begin{equation*}
\varepsilon_n(x,r) =  \frac1{r^{n}}\inf_H\sum_{i=+,-}\mathcal{H}^{n}\big(S_{H}^i\setminus \Omega^i\big).
\end{equation*}
\begin{theorem}
Let $\Omega^+,\Omega^-\subset\mathbb{R}^{n+1}$ be open and disjoint and let
\begin{equation*}
E:= \left\{x\in\mathbb{R}^{n+1}:\int_0^1 \varepsilon_n(x,r)^2 \, \frac{dr}{r}<\infty\right\}\, .
\end{equation*}
Then $E$ is $n$-rectifiable. 
\end{theorem}

The strategy used in \cite{FTV} is to prove that every subset of this set with finite measure is rectifiable. In order to conclude that the whole set has $\sigma$-finite measure (and hence is $n$-rectifiable) \cite{FTV} uses a specific argument which still relies on Lemma \ref{l:Rogers}. Alternatively we can now use Theorem \ref{t:main} and Corollary \ref{c:main}. We only need to check that $A_{\varepsilon_n}$ is Souslin, but indeed it is Borel. This is already remarked in \cite{FTV}, but we include a brief discussion here. Indeed the coefficients $\varepsilon_n(x,r)$ are upper semicontinous on $(x,r)$ and thus, for every positive $s$, the map $x\mapsto \int_s^1 \varepsilon_n(x,r)^2\, \frac{dr}{r}$ is upper semicontinous. This is enough to check that the $A_{\varepsilon_n}$ is Borel. 

\subsection{Acknowledgments} Both authors acknowledge the support of the National Science Foundation through the grant FRG-1854147. The second author wishes to thank Tuomas Orponen, Pablo Shmerkin, Xavier Tolsa, and Michele Villa for some really useful conversations about the question.

\section{Geometric preliminaries}

We denote by $\mathcal{Q}_j$ the family of closed dyadic cubes of $\mathbb R^n$ with sidelength $2^{-j}$.

\begin{definition}\label{d:holes}
We say that a set $E$ is sparse if there is a sequence of integers $l_j\uparrow \infty$ with the following property. For every $Q\in \mathcal{Q}_{l_j}$ there is a subcube $Q'\subset Q$ with 
$Q'\in \mathcal{Q}_{l_j+j}$ such that 
\begin{equation}\label{e:holes}
E\subset \bigcup_{Q\in \mathcal{Q}_{l_j}} Q'\, .
\end{equation}
\end{definition}

For the purpose of Theorem \ref{t:main} the latter definition would be enough and we would need the rather standard fact that any subset of a sparse set cannot be $k$-rectifiable for any $k\geq 1$. In order to handle the more general case of doubling metric spaces it is however useful to consider a version of the sparsity defined above which uses balls in place of cubes. 

\begin{definition}\label{d:ballsholes}
We say that a set $E$ is weakly sparse if there is a sequence of real numbers $r_j \downarrow 0$ and a universal constant $N$ with the following property. For every ball $B$ of radius $r_j$ there exist at most $N$ balls $\{B_l\}_{1 \leq l \leq N}$ of radius $2^{-j}r_j$ such that
\begin{equation}\label{e:holes2}
E \cap B \subset \bigcup_{1 \leq l \leq N} B_l\, .
\end{equation}
\end{definition}

\begin{remark}
Notice that if $E$ is sparse then it is weakly sparse with $N=3^n$. This follows from the fact that a ball of radius $2^{-j+1}$ can intersect $3^n$ cubes of sidelength $2^{-j}$.
\end{remark}

The main tool we need from this section towards Theorem \ref{t:main} is the following proposition.

\begin{proposition}\label{p:holes}
Let $k$ be a positive integer and $(X,d)$ a metric space. If $F\subset X$ is an $\mathcal{H}^k$-measurable set which is weakly sparse, then $F$ is purely $k$-unrectifiable. 
\end{proposition}

\begin{proof}[Proof of Proposition \ref{p:holes}]
We can reduce to the case where $X$ is Banach space by embedding the metric space $X$ isometrically into some Banach space.

Suppose that $F$ is not purely $k$-unrectifiable. This means there exists a Borel set $\mathbb B\subset \mathbb R^k$ and a Lipschitz function $f: \mathbb B \rightarrow X$ such that $\mathcal{H}^k(f(\mathbb B) \cap F)>0$.  We will show that this cannot happen. It will be enough to show this to be the case when $f$ is a bi-Lipschitz map on $\mathbb B$. Indeed, a rectifiable subset of a Banach space can be covered, up to a set of $\mathcal{H}^k$ measure zero, by a countable union of bi-Lipschitz images of Borel sets. We appeal to chapter 7 of  \cite{MattilaSurvey}, page 56, for this particular characterization of rectifiable subsets of Banach spaces.

We will denote the bi-Lipschitz constant by $L$. We can further restrict $\mathbb B$ to an $\mathcal{H}^k$-measurable subset, such that its image via $f$ is now contained in $F$. We claim that $\mathbb B$ must be a weakly sparse subset of $\mathbb R^k$. Let $\hat{B}$ be any ball of radius ${(2L)}^{-1}r_j$.

Then $f(\hat{B}) \subseteq B$ where $B$ is a ball of radius $r_j$. This means that there exist $\{B_l\}_{1 \leq l \leq N}$ balls of radius $2^{-j}r_j$ such that
\begin{equation*}
\mathbb B \cap \hat {B} \subseteq \bigcup_{1 \leq l \leq N}f^{-1}(B_l).
\end{equation*}
There exists balls $\hat{B_l}$ of radius $L2^{-j}r_j$ that contain the ball $f^{-1}(B_l) \subset \hat{B_l}$. This is indeed the weakly sparse condition.

Now the weakly sparse condition will give us that for every $x \in \mathbb B$, the lower Lebesgue density is zero (i.e. $\Theta^k_{*}(x,\mathbb B)=0$). This implies that $\mathbb B$ has zero Lebesgue measure, and thus $f(\mathbb B)$ has zero $\mathcal{H}^k$ measure. This completes the proof of the proposition.
\end{proof}

\section{Generalized Hausdorff measures}

In this section we collect the other tool which is needed to prove Theorem \ref{t:main}. We follow \cite{Rogers} and introduce the following generalization of the Hausdorff measure $\mathcal{H}^k$. 

\begin{definition}\label{d:generalized-Hausdorff}
Given a continuous increasing function $r\mapsto h(r)$ with $h(0)=0$ and a set $E$ we define
\begin{equation*}
\HH_{\delta}^{h} (E) = \inf \left\{ \left. \sum_{i = 1}^{\infty} h \left( \mathrm{diam} (C_{i}) \right) \right| \mathrm{diam} (C_{i}) \leq \delta, \bigcup_{i = 1}^{\infty} C_{i} \supseteq E \right\}.
\end{equation*}
We also define $\HH^h(E)=\lim_{\delta \rightarrow 0} \HH_{\delta}^h(E)$.
\end{definition}

In particular the usual Hausdorff measure $\mathcal{H}^k$ amounts to the choice of $h (r) = \omega_k \left(\frac{r}{2}\right)^k$. 

We will need three key lemmas. The following is a generalization of Frostman's lemma.

\begin{lemma}\label{l:Frostman}
Let $r\mapsto h (r)$ be a continuous increasing function with $h(0)=0$. Given an analytic set $E$ such that $\mathcal{H}^h (E)>0$ there
is a Radon measure $\mu$ such that $\mu (E) >0$ and $\mu (Q) \leq h(\diam (Q))$ for every dyadic cube $Q$. 
\end{lemma}
\begin{proof} We can use \cite[Corollary 2, Section 6.2]{Rogers} to find an $F_\sigma$ set $E'\subset E$ with $\mathcal{H}^h (E')=\mathcal{H}^h (E)$, which in turn reduces the lemma to the case of a compact $E$. We can then follow the proof of \cite[Theorem 8.8]{Mattila}. We report it here for the reader's convenience to highlight that, even though the proof in \cite[Theorem 8.8]{Mattila} is written for the case $h (r) = \omega_s \left(\frac{r}{2}\right)^s$, namely that of classical Hausdorff measures, it works for general $h$. 

Without loss of generality we assume that $E\subset [0,1]^n$. We fix a positive integer $m$ and define the measure $\mu_{m,m}$ as 
\begin{itemize}
\item $\mu_{m,m} \res Q = h (\diam (Q)) 2^{-n m}\mathscr{L}^n \res Q$ for every $Q\in \mathcal{Q}_m$ such that $Q\cap E\neq \emptyset$;
\item $\mu_{m,m} \res Q =0$ otherwise.
\end{itemize}
We then define $\mu_{m, l}$ inductively with $l$ decreasing till $\mu_{m,0}$. Assuming that $\mu_{m,l}$ has been defined and $l>0$, we define $\mu_{m, l-1}$ on each cube $Q\in \mathcal{Q}_{l-1}$ as follows:
\begin{itemize}
\item $\mu_{m, l-1} \res Q = \mu_{m, l} \res Q$ if $\mu_{m,l} (Q) \leq h (\diam (Q))$;
\item $\mu_{m, l-1} \res Q = h (\diam (Q)) (\mu_{m,l} (Q))^{-1} \mu_{m,l} \res Q$ otherwise.
\end{itemize}
We then set $\mu_m := \mu_{m,0}$ and observe that $\mu_{m,0} (\mathbb R^n) \leq h (\diam ([0,1]^n)) = h (\sqrt{n})$. In particular we can assume the existence of a subsequence, not relabeled, converging weakly$^*$ to some Radon measure $\mu$. Observe that 
$\mu_m (Q) \leq h (\diam (Q))$ for every dyadic cube $Q$ with $\diam (Q) \geq \sqrt{n} 2^{-m}$. However, since $Q$ is closed, we have the inequality $\mu (Q) \geq \limsup_{m\uparrow \infty} \mu_m (Q)$, which of course cannot be used to get an upper bound on $\mu (Q)$. Rather, for every fixed $Q$, consider the $3^n$ dyadic cubes $Q_1, \ldots , Q_{3^n}$ with same diameter which have nonempty intersection with it. If we let $U$ be the interior of $\bigcup_i Q_i$, then obviously 
$\mu_m (U) \leq 3^n h (\diam (Q))$ whenever $\sqrt{n} 2^{-m} \leq \diam (Q)$. On the other hand $U$ is open and contains $Q$, so we can conclude
\[
\mu (Q) \leq \mu (U) \leq \liminf_{m\to \infty} \mu_m (Q) \leq 3^n h (\diam (Q))\, .
\]
In particular, up to dividing by $3^n$, we achieve the inequality $\mu (Q) \leq h (\diam (Q))$ for every dyadic $Q$.

Next consider that $\mu_m$ is, by construction, supported in the union of cubes of diameter $\sqrt{n} 2^{-m}$ which intersect $E$, and thus in the $\sqrt{n} 2^{-m}$-neighborhood of $E$. Since $E$ is closed, $\mu$ is therefore supported in $E$. Hence, as soon as we show that $\mu (\mathbb R^n) > 0$, we are done. On the other hand $\mu$ is supported in $[0,1]^n$ and therefore $\mu (\mathbb R^n) = \lim_m \mu_m (\mathbb R^n)$. It thus suffices to find a lower bound for $\mu_m (\mathbb R^n)$. Fix $m$ and observe that each $x\in E$ is contained in a cube $Q$ of diameter at least $\sqrt{n} 2^{-m}$ such that $\mu_m (Q) = h (\diam (Q))$ and $Q\subset [0,1]^n$. Denote it by $Q_x$. $\{Q_x\}_{x\in E}$ is then a finite collection of dyadic cubes. So, given two different $x$ and $y$ in $E$, either $Q_x\subset Q_y$, or the intersection of $Q_x$ and $Q_y$ is contained in $\partial Q_x$. In particular if we let $\mathcal{C}$ denote the family of those $Q_x$ which are maximal under inclusion, we find that the elements of $\mathcal{C}$ have disjoint interiors, while they form a covering of $E$. Note that $\mu_m$ is absolutely continuous with respect to the Lebesgue measure: though two distinct cubes $Q, Q'\in \mathcal{C}$ might have nonempty intersection, the latter has zero Lebesgue measure, and hence zero $\mu_m$ measure. We therefore have 
\[
\sum_{Q\in \mathcal{C}} h (\diam (Q)) = \sum_{Q\in \mathcal{C}} \mu_m (Q) \leq \mu_m (\mathbb R^n)\, .
\]   
However, because $\mathcal{C}$ is a cover of $E$, we find
\[
\mathcal{H}^h_\infty (E) \leq \sum_{Q\in \mathcal{C}} h (\diam (Q)) \leq \mu_m (\mathbb R^n)\, .
\]
Now, since $\mathcal{H}^h (E) >0$, we must have $\mathcal{H}^h_\infty (E) >0$ as well, in particular we have found the desired positive lower bound for $\mu_m (\mathbb R^n)$.
\end{proof}

The second is a classical result for ``usual'' Hausdorff measures, which follows from combining \cite[Theorem 8.8]{Mattila} with \cite[Theorem 8.13]{Mattila}.

\begin{lemma}\label{l:inverse-Frostman}
Assume $\nu$ is a (nontrivial) Radon measure on $\mathbb R^n$ and $s > 0$ a real number such that $\nu (Q) \leq  (\diam (Q))^s$ for every dyadic cube $Q$. Then there is a closed subset $F\subset {\rm spt}\, (\nu)$ such that $0< \mathcal{H}^s (F) < \infty$. 
\end{lemma}

The third is perhaps the most critical ingredient. In the case of a closed set $E$ it is a theorem of Rogers in \cite{Rogers2}, generalizing a classical one by Besicovitch in the plane, cf. \cite[Theorem 3]{Besicovitch}. The case of an analytic set is achieved combining Roger's Theorem with \cite[Theorem 6.6]{SS}, which implies that any analytic set which is not $\mathcal{H}^h$ $\sigma$-finite contains a compact set which is not $\mathcal{H}^h$-$\sigma$-finite (the case $h (r)=r^s$ of the latter theorem is due to Davies in \cite{Davies}). 

\begin{lemma}\label{l:Rogers}
Let $r\mapsto h (r)$ be a continuous increasing function with $h(0)=0$ and let $E\subset \mathbb R^n$ be an analytic set which is not $\mathcal{H}^h$-$\sigma$-finite. Then there is a continuous function $r\mapsto g (r)$ with $g (0)=0$ such that  $E$ is still not $\mathcal{H}^{g}$-$\sigma$-finite and 
\begin{equation}\label{e:vanishing}
\lim_{r\downarrow 0} \frac{g(r)}{h (r)} = 0\, .
\end{equation}
\end{lemma} 

\section{Proof of Theorem \ref{t:main}}

We start with a set $E$ as in Theorem \ref{t:main} and we apply Lemma \ref{l:Rogers} to find a continuous function $h$ such that $h (0)=0$, $r\mapsto h (r)$ is increasing, 
\begin{equation}\label{e:vanishing-1}
\lim_{r\downarrow 0} \frac{h(r)}{r^k}=0\, ,
\end{equation}
and $E$ is not $\mathcal{H}^h$ $\sigma$-finite. We next apply Lemma \ref{l:Frostman} to find a corresponding measure $\mu$. We can, without loss of generality, assume $\mu (\mathbb R^n\setminus E)=0$, since we can restrict $\mu$ to $E$ (note that, since $E$ is analytic, $E$ is also universally measurable, which in particular implies that it is $\mu$-measurable). Since $\mu$ is a Radon measure, we can find a compact subset $E'$ such that $\mu (E')>0$. Without loss of generality we assume that $\mu$ is in fact supported in $E'$ (again, it suffices to substitute $\mu$ with $\mu \res E'$). We next prove the following lemma.

\begin{lemma}\label{l:key-lemma}
Let $h$ be a continuous increasing function such that \eqref{e:vanishing-1} holds and $\mu$ a nontrivial Radon measure such that $\mu (Q) \leq h (\diam (Q))$ for every dyadic cube $Q$. For any given integer $\ell$ there is nontrivial Radon measure $\nu$ with the following properties:
\begin{itemize}
\item[(a)] $\nu (Q) \leq (\diam (Q))^k$ for every dyadic cube $Q$;
\item[(b)] the support of $\nu$ is contained in the support of $\mu$;
\item[(c)] the support of $\nu$ is sparse.
\end{itemize}
\end{lemma} 

Having Lemma \ref{l:key-lemma} at our disposal, we find a corresponding measure $\nu$. We can then apply Lemma \ref{l:inverse-Frostman} to find a set $F$ in the support of $\nu$ such that $0< \mathcal{H}^k (F) < \infty$. It follows from our discussion so far that $F$ is contained in the set $E$ given at the very beginning. On the other hand such set $F$ must be sparse and hence, by Proposition \ref{p:holes}, it is purely $k$-unrectifiable, thus completing the proof of Theorem \ref{t:main}. We are thus left with proving Lemma \ref{l:key-lemma}.

\begin{proof} We fix $\mu$ as in the statement and, without loss of generality, we assume that $\mu$ is supported in $[0,1]^n$ and that $\mu$ is a probability measure. We also assume, without loss of generality, that 
\begin{equation}\label{e:annoying}
\mu (\partial Q) = 0 \qquad \mbox{for every dyadic cube $Q$.}
\end{equation}
In fact, if the condition \eqref{e:annoying} does not hold, it means that $\mu (F) >0$ for some face $F$ of $\partial Q$. We can then restrict $\mu$ to the latter face, and reduce the statement to the same one in $\mathbb R^{n-1}$. We can iterate this procedure and note that however it has to stop before lowering the dimension of the ambient space to $k$, since obviously the support of $\mu$ is not $\mathcal{H}^k$ $\sigma$-finite. 

\medskip

The measure $\nu$ will be constructed as weak$^*$ limit of an appropriate sequence $\nu_j$ of measures. The latter, which will be constructed inductively from $\mu = \nu_0$, will all be probability measures satisfying the following properties.
\begin{itemize}
\item[(i)] The support of $\nu_j$ is contained in the support of $\nu_{j-1}$ and \eqref{e:annoying} holds with $\nu_j$ in place of $\mu$. 
\item[(ii)] $\nu_j (Q) \leq \min \{2^{j^2 n} h (\diam (Q)), C_0 (\diam (Q))^k\}$ for all dyadic cubes $Q$, where the constant $C_0\geq 1$ is chosen so that $h(r) \leq C_0 r^k$ for every $r\leq \sqrt{n}$. 
\end{itemize}
Moreover, there will be a sequence of integers $l_j$ satisfying $l^1_{j+1}\geq l_j+j$ such that $\nu_j$ satisfies the following additional properties:
\begin{itemize}
\item[(iii)] $\nu_j (Q) = \nu_{j-1} (Q)$ for all dyadic cubes $Q\in \mathcal{Q}_{l_j}$.
\item[(iv)] For $j\geq 1$ there is a family $\mathcal{F}_j \subset \mathcal{Q}_{l_j +j}$ such that
\begin{itemize}
\item[(iv')] $\nu_j$ is supported in $\bigcup_{Q'\in \mathcal{F}_j} Q'$;
\item[(iv'')] For each cube $Q\in \mathcal{Q}_{l_j}$ there is at most one cube $Q'\in \mathcal{F}_j$ such that $Q'\subset Q$.
\end{itemize}
\end{itemize}
Observe that, combining (iii) and (i) we actually know that
\begin{itemize}
\item[(iii')] $\nu_j (Q) = \nu_{j-1} (Q)$ for all dyadic cubes $Q\in \mathcal{Q}_l$ with $l\leq l_j$.
\end{itemize}
Moreover $\nu_0=\mu$ satisfies all the requirements.

Before coming to the inductive construction of $\nu_j$, let us observe that $\nu_j$ converges weakly$^*$ to some measure $\nu$ (we could in fact appeal to the weak$^*$ compactness of probability measures to conclude the existence of such a limit for an appropriate subsequence, but actually by (iii') there is no need for such an extraction) and that such $\nu$ satisfies the requirements of the lemma. In fact ${\rm spt}\, (\nu) \subset {\rm spt}\, (\nu_j)$ for every $j$ by (i), which in particular implies that ${\rm spt}\, (\nu)\subset {\rm spt}\, (\nu_0)= {\rm spt}\, (\mu)$. Moreover, because of (iv), we conclude that ${\rm spt}\, (\nu)$ is $\ell$-sparse. Observe that, because of (ii), $\nu_j ([0,1]^n)= \mu ([0,1]^n) = 1$ for every $j$, and in particular $\nu$ must be a probability measure. Finally, from (ii) and (iii) we also conclude (a) after rescaling. 

\medskip

We now come to the construction of $\nu_j$ from $\nu_{j-1}$.  First we pick $l_j \geq l_{j-1} + j$ so that 
\begin{equation}\label{e:condition-1}
\frac{h (\diam (Q))}{(\diam (Q))^k} \leq 2^{-n j^2} \qquad \forall Q\in \mathcal{Q}_l \mbox{ with $l\geq l_j$,}
\end{equation} 
which obviously is possible because of \eqref{e:vanishing-1}. 
Next, select all the cubes $Q\in \mathcal{Q}_{l_j}$ for which $\nu_{j-1} (Q)>0$ and call this collection $\mathscr{G}$. Secondly, subdivide each cube $Q$ into $2^{nj}$ cubes belonging to $\mathcal{Q}_{l_j + j}$ and select one $Q'$ among them so that
\[
\nu_{j-1} (Q') \geq 2^{-nj} \nu_{j-1} (Q)\, .
\] 
We next define $\mathcal{F}_j$ to be the collection of such $Q'$ and 
and we define $\nu_j$ on each cube $Q\in \mathscr{G}$ as 
\begin{equation}\label{e:reduced}
\nu_j \res Q = \frac{\nu_{j-1} (Q)}{\nu_{j-1} (Q')} \nu_{j-1} \res Q'\, .
\end{equation}
Note that the latter gives a well defined measure because we know $\nu_{j-1} (\partial Q)=0$ for every dyadic cube $Q$. 

Observe that (i), (iii), and (iv) all hold by construction. As for (ii), because of (iii'), the inequality holds for any cube $Q\in \mathcal{Q}_l$ with $l\leq l_j$. We thus need to show it for any cube $Q\in \mathcal{Q}_l$ with $l>l_j$. But by construction such a cube must be contained in an element $\bar{Q}\in \mathscr{G}$. If we denote by $\bar Q'$ the element of $\mathcal{F}_j$ contained in $\bar Q$, \eqref{e:reduced} yields the inequality 
\[
\nu_j (Q) \leq  \frac{\nu_{j-1} (\bar Q)}{\nu_{j-1} (\bar Q')} \nu_{j-1} (Q) \leq 2^{nj} \nu_{j-1} (Q)\, . 
\] 
Next, observe that, since (iii) holds for $\nu_{j-1}$, we have  
\begin{align*}
\nu_j (Q) &\leq 2^{nj} \min \{ 2^{(j-1)^2 n} h (\diam (Q)), C_0 (\diam (Q))^k\} \\
& \leq
2^{j^2n} h (\diam (Q)) = \min \{ 2^{j^2n} h (\diam (Q)), C_0 (\diam (Q))^k\}\, ,
\end{align*}
where we have used \eqref{e:condition-1}, which we can apply because $Q\in \mathcal{Q}_l$ and $l>l_j$. 
In particular this shows the estimate (ii) for any dyadic cube, and hence completes the proof. 
\end{proof} 
\begin{proof}[Sketch of the proof of Remark \ref{r:metric}]
We provide a brief sketch of the modifications to be made. The interested reader can fill the details. 

First of all, any $\sigma$-compact metric space $X$ can be written as the countable union of compact subsets $X_i$ of $X$, while if $E\subset X$ is not $\sigma$-finite for the Hausdorff $k$-dimensional measure, then for some $i$ $E\cap X_i$ must be non $\sigma$-finite for the Hausdorff $k$-dimensional measure on $X_i$. This shows that our statement reduces to compact metric spaces. 

Lemma \ref{l:Rogers} can be used in the same way as before, in particular in the paper \cite{SS} the lemma is proved in compact metric spaces.  

The proofs of Lemma \ref{l:key-lemma} and Lemma \ref{l:Frostman} rely on the dyadic cube structure of the Euclidean space, which still exists on doubling metric spaces. Note that the constant $N$ will depend precisely on the doubling constant of the metric space.

The dyadic cube structure of doubling metric spaces has been frequently exploited in harmonic analysis and there is much literature on the topic. The reader can consult for instance \cite{HK}. 

We can then complete the proof using Proposition \ref{p:holes}. 
\end{proof}
\bibliographystyle{plain}

\end{document}